\documentclass[12pt]{article}
\usepackage{amsmath}
\usepackage{latexsym}
\usepackage{amssymb}
\newtheorem{thm}{Theorem}[section]
\newtheorem{la}[thm]{Lemma}
\newtheorem{Defn}[thm]{Definition}
\newtheorem{Remark}[thm]{Remark}
\newtheorem{Note}[thm]{Note}
\newtheorem{prop}[thm]{Proposition}

\newtheorem{cor}[thm]{Corollary}
\newtheorem{Example}[thm]{Example}
\newtheorem{Examples}[thm]{Examples}
\newtheorem{Problems}[thm]{Problems}

\newtheorem{Problem}[thm]{Problem}
\newtheorem{Convention}[thm]{Convention}
\newtheorem{Number}[thm]{\!\!}
\newenvironment{defn}{\begin{Defn}\rm}{\end{Defn}}

\newenvironment{rem}{\begin{Remark}\rm}{\end{Remark}}
\newenvironment{numba}{\begin{Number}\rm}{\end{Number}}
\newenvironment{proof}{{\noindent\bf Proof.}}%
                  {\nopagebreak\hspace*{\fill}$\Box$\medskip\medskip\par}   
\newcommand{\Punkt}{\nopagebreak\hspace*{\fill}$\Box$}
\newcommand{\wb}{\overline}
\newcommand{\ve}{\varepsilon}
\newcommand{\at}{\symbol{'100}}

\newcommand{\wt}{\widetilde}
\newcommand{\tensor}{\otimes}
\newcommand{\impl}{\Rightarrow}

\newcommand{\mto}{\mapsto}

\newcommand{\N}{{\mathbb N}}
\newcommand{\R}{{\mathbb R}}

\newcommand{\K}{{\mathbb K}}

\newcommand{\Z}{{\mathbb Z}}

\newcommand{\sub}{\subseteq}

\DeclareMathOperator{\im}{im}

\DeclareMathOperator{\id}{id}

\newcommand{\cA}{{\cal A}}

\newcommand{\obs}{{\footnotesize\rm s}}
\newcommand{\obc}{{\footnotesize\rm c}}
\newcommand{\obu}{{\footnotesize\rm u}}
\newcommand{\obcs}{{\footnotesize\rm cs}}

\DeclareMathOperator{\Lip}{Lip}

\begin{document}
\begin{center}
{\Large\bf Invariant manifolds for\vspace{2mm} finite-dimensional
non-archimedean dynamical systems}\\[6mm]
{\bf Helge Gl\"{o}ckner}\vspace{4mm}
\end{center}
\begin{abstract}
\hspace*{-7.2 mm}
Let $M$ be an analytic manifold modelled
on an ultrametric Banach space
over a complete ultrametric field $\K$.
Let $f\colon M \to M$ be an analytic
diffeomorphism and $p$ be a fixed point of $f$.
We discuss invariant manifolds around $p$, like
stable manifolds,
centre-stable manifolds and
centre manifolds, with an emphasis on results
specific to the case that $M$ has finite dimension.
The results have applications
in the theory of Lie groups
over totally disconnected local fields.\vspace{3mm}
\end{abstract}
{\footnotesize {\em Classification}:
37D10 (Primary) 
46S10, 
26E30 (Secondary)\\[2.5mm]
}\vspace{6mm}
\begin{center}
{\bf\large Introduction and statement of main results}
\end{center}
Guided by the classical theory
of invariant manifolds for time-discrete
smooth dynamical
systems over the real ground field (cf.\ \cite{HaK}, \cite{HPS}, \cite{Ir1},
\cite{Wel}),
invariant manifolds have recently also been
constructed for time-discrete
analytic dynamical systems over a complete ultrametric field $(\K,|.|)$  \cite{EXP}.
The invariant manifolds are useful in the theory of Lie groups
over local fields, where they allow results to be extended
to ground fields of positive characteristic,
which previously where available only in characteristic $0$
(i.e., for $p$-adic Lie groups).
To enable these Lie theoretic applications,
the general theory from~\cite{EXP}
is not sufficient, and additional,
more specific results concerning ultrametric invariant
manifolds are needed.
The goal of this article is to provide such
complementary results, including simplifications of the
theory from \cite{EXP} for finite-dimensional dynamical systems,
which make it applicable in the situations at hand.\\[2.5mm]
As in the real case, hyperbolicity assumptions
are essential for a discussion of invariant manifolds.
Roughly speaking, a continuous linear self-map $\alpha\colon E\to E$
of an ultrametric Banach space $E$ over $\K$
is called hyperbolic if $E$ admits a decomposition $E=E_\obs\oplus E_\obu$
into a stable subspace $E_\obs$ on which $\alpha$ is contractive
and an unstable subspace $E_\obu$
on which $\alpha$ is expansive.
More precisely,
$\alpha$ is called \emph{hyperbolic} if
it is $1$-hyperbolic
in the following sense~\cite{EXP}:\\[2.5mm]
{\bf Definition.}
The continuous linear map
$\alpha\colon E\to E$ is said to be \emph{$a$-hyperbolic}
for $a\in \; ]0,\infty[$
if there exist $\alpha$-invariant
vector subspaces $E_{a,\obs}$ and $E_{a,\obu}$ of $E$ such that
$E=E_{a,\obs} \oplus E_{a,\obu}$,
and an ultrametric norm $\|.\|$ on $E$ defining its topology,
with properties (a)--(c):
\begin{itemize}
\item[(a)]
$\|x+y\|=\max\{\|x\|,\|y\|\}$
for all $x\in E_{a, \obs}$ and $y\in E_{a,\obu}$;
\item[(b)]
$\alpha_2:=\alpha|_{E_{a,\obu}}$
is invertible;
\item[(c)]
$\|\alpha_1\|< a$ and $\frac{1}{\|\alpha_2^{-1}\|}>a$
holds for the operator norms with respect to $\|.\|$,
where $\alpha_1:=\alpha|_{E_{a,\obs}}$
(and $\frac{1}{0}:=\infty$).
\end{itemize}
Then $E_{a,\obs}$ is uniquely determined and if $\alpha$ is invertible
or $E$ finite-dimensional,
then also $E_{a,\obu}$ is unique (see \cite[Remark 6.4]{EXP}
and Remark~\ref{centunique} below).
If $a=1$, we also write $E_\obs:=E_{1,\obs}$
and $E_\obu:=E_{1,\obu}$.\\[2.5mm]
Similarly, $E$ may have an $a$-centre-stable subspace
$E_{a,\obcs}$ such that
\[
E=E_{a,\obcs}\oplus E_{a,\obu},
\]
or an $a$-centre subspace $E_{a,\obc}$ such that
\[
E=E_{a,\obs}\oplus E_{a,\obc}\oplus E_{a,\obu};
\]
see Definitions~\ref{defcsub} and \ref{defacentre} for details.
We omit the subscript ``$a$''
if $a=1$.\\[3.5mm]
It is useful to fix a notation
for the set of absolute values
of eigenvalues,
in the finite-dimensional case.\\[2.5mm]
{\bf Definition.}
Let $\alpha\colon E\to E$ be a linear
self-map of a
finite-dimensional
vector space $E$ over a complete ultrametric field~$(\K,|.|)$.
We use the same symbol, $|.|$, for the unique
extension of $|.|$ to an absolute
value on an algebraic closure $\wb{\K}$
of $\K$ (see \cite[Theorem 16.1]{Sch}).
We write $R(\alpha)\sub [0,\infty[$ for
the set of all $|\lambda|$
such that
$\lambda\in \wb{\K}$ is an eigenvalue
of $\alpha\tensor_\K\id_{\wb{\K}}$.\\[3.5mm]
The above definition of hyperbolicity
is a good basis for theorems,
but may be difficult to verify
directly.
Fortunately,
in the finite-dimensional case,
an easier (and more concrete)
description of hyperbolicity
can be obtained.
Also, the existence
of centre subspaces and centre-stable subspaces
is automatic:\\[3.5mm]
{\bf Theorem A.}
\emph{Let $\alpha\colon E\to E$ be a linear
self-map of a
finite-dimensional
vector space $E$ over a complete ultrametric field~$\K$.
Then $E$ admits
an $a$-centre-stable subspace and
an $a$-centre subspace,
for each $a\in\;]0,\infty[$.
Moreover,
$\alpha$ is $a$-hyperbolic
if and only if $a\not\in R(\alpha)$.}\\[3.5mm]
Let $M$ be an analytic manifold modelled on an ultrametric Banach
space $E$ over $\K$ (as in \cite{Bo1}).
An analytic diffeomorphism $\kappa\colon U\to V$ from an
open set $U\sub M$ onto an open set $V\sub E$ is called a \emph{chart}
for~$M$.
An analytic map $f\colon N\to M$
between analytic manifolds
is called an \emph{immersion}
if, for each $x\in N$,
the tangent map $T_x(f)\colon T_x(N)\to T_{f(x)}(M)$
is a homeomorphism onto its image
$\im T_x(f)$, and $\im T_x(f)$
is complemented in $T_{f(x)}(M)$ as
a topological vector space.
If $M$ and $N$ have finite dimension,
this simply means that $T_x(f)$ is injective
for each $x\in N$.
An analytic manifold $N$ is called an \emph{immersed submanifold}
of $M$ if $N\sub M$ as a set and the inclusion map $\iota\colon
N\to M$ is an immersion.
For $x\in N$, we
identify $T_x(N)$ with the vector subspace $\im T_x(\iota)$
of~$T_x(M)$.\\[2.5mm]
As before, let $M$ be an analytic manifold
modelled on an ultrametric Banach spaces $E$ over a
complete ultrametric field $(\K,|.|)$.
Let $f\colon M\to M$
be an analytic diffeomorphism,
and $p\in M$ be a fixed point of $f$.\\[3.5mm]
%
%
%
{\bf Definition.}
Given $a\in \;]0,1]$, we define $W_a^\obs(f,p)\sub M$,
the \emph{$a$-stable set} around $p$ with respect to $f$,
as the set of all $x\in M$ such that
%
\begin{equation}\label{dewaseq}
\mbox{$f^n(x)\to p\;$  as $\;n\to\infty\;$
and $\; a^{-n}\|\kappa(f^n(x))\|\to 0\, $,}
\end{equation}
for some (and hence every) chart $\kappa\colon U\to V\sub E$
of $M$ with $p\in U$ such that $\kappa(p)=0$,
and some (and hence every) ultrametric norm $\|.\|$ on $E$
defining its topology.\footnote{See \cite[Remark 6.5]{EXP}
for the independence of the choice of $\kappa$ and $\|.\|$.}\\[3.5mm]
It is clear from the definition that $W_a^\obs:=W_a^\obs(f,p)$
is stable under $f$, i.e., $f(W_a^\obs)=W_a^\obs$.
If the tangent map $T_p(f)\colon T_p(M)\to T_p(M)$
is $a$-hyperbolic
(which can be checked using Theorem~A),
then $W_a^\obs$ is an analytic manifold,
the \emph{$a$-stable manifold}
around $p$ with respect to $f$:\\[3.5mm]
{\bf Ultrametric Stable Manifold Theorem} (cf.\ \cite[Theorem 1.3]{EXP}).
\emph{Let $M$ be an analytic manifold
modelled on an ultrametric Banach space over
a complete ultrametric field $\K$.
Let $f\colon M\to M$ be an analytic diffeomorphism,
$p\in M$ be a point fixed by $f$, and
$a\in \;]0,1]$. If the tangent map
$\alpha:=T_p(f)\colon T_p(M)\to T_p(M)$ is $a$-hyperbolic}
(\emph{which is satisfied
if $M$ is finite-dimensional
and $a\not\in R(\alpha)$}),
\emph{then there exists a unique analytic manifold structure
on $W_a^\obs:=W_a^\obs(f,p)$ such that} (a)--(c) \emph{hold}:
\begin{itemize}
\item[\rm(a)]
\emph{$W_a^\obs$ is an immersed submanifold of $M$};
\item[\rm(b)]
\emph{$W_a^\obs$ is tangent to the $a$-stable subspace $T_p(M)_{a, \obs}$}
(\emph{with respect to $T_p(f)$}), \emph{i.e.,
$T_p(W_a^\obs)= T_p(M)_{a,\obs}$};
\item[\rm(c)]
\emph{$f$ restricts to an analytic diffeomorphism $W_a^\obs\to W_a^\obs$.}
\end{itemize}
\emph{Moreover, each neighbourhood of $p$ in $W_a^\obs$ contains
an open neighbourhood $\Omega$ of $p$ in $W_a^\obs$
which is a submanifold of $M$,
is $f$-invariant $($i.e., $f(\Omega)\sub \Omega)$,
and such that
$W_a^\obs =\bigcup_{n=0}^\infty f^{-n}(\Omega)$.}\\[3.5mm]
If $T_p(f)$ is hyperbolic, then $W_1^\obs$ is simply called the \emph{stable manifold}
around~$p$, and denoted~$W^\obs$.\\[2.5mm]
Now consider the following local situation:\\[2.5mm]
Let $M$
be an analytic manifold modelled on an ultrametric Banach space
over a complete ultrametric field~$\K$.
Let $M_0\sub M$ be open,
$f\colon M_0 \to M$ be an analytic mapping,
$p\in M_0$ be a fixed point of $f$,
and $a\in \;]0,1]$. The following
four definitions are taken from~\cite{EXP}.\\[2.5mm]
{\bf Definition.}
If $T_p(M)$ has an $a$-centre-stable subspace $T_p(M)_{a,\obcs}$ with
respect to $T_p(f)$,
we call an immersed submanifold $N\sub M_0$
an \emph{$a$-centre-stable manifold} around~$p$ with respect to~$f$
if (a)--(d) are satisfied:
\begin{itemize}
\item[(a)]
$p\in N$;
\item[(b)]
$N$ is tangent to
$T_p(M)_{a,\obcs}$ at $p$, i.e.,
$T_p(N)=T_p(M)_{a,\obcs}$;
\item[\rm(c)]
$f(N)\sub N$; and
\item[\rm(d)]
$f|_N\colon N\to N$ is analytic.
\end{itemize}
If $a=1$, we simply speak of a \emph{centre-stable manifold}.\\[2.5mm]
{\bf Definition.}
If $T_p(f)$ is an automorphism
and $T_p(M)$ has a centre subspace $T_p(M)_\obc$
with respect to $T_p(f)$,
we say that an immersed submanifold $N\sub M_0$
is a \emph{centre manifold} around $p$ with respect to $f$
if (a), (c) and (d) from the preceding definition hold as well as
\begin{itemize}
\item[(b)$'$]
$N$ is tangent to
$T_p(M)_\obc$ at~$p$, i.e.,
$T_p(N)=T_p(M)_\obc$.
\end{itemize}
{\bf Definition.}
In the situation above, assume that $T_p(f)$ is $a$-hyperbolic.
An immersed submanifold $N\sub M_0$
is called a \emph{local $a$-stable manifold} around~$p$ with respect to~$f$
if (a), (c) and (d) just stated are satisfied
as well as
\begin{itemize}
\item[(b)$''$]
$N$ is tangent at $p$ to
the $a$-stable subspace $T_p(M)_{a,\obs}$ with respect to $T_p(f)$, i.e.,
$T_p(N)=T_p(M)_{a,\obs}$.
\end{itemize}
If $a=1$, we simply speak of a \emph{local stable manifold}.\\[2.5mm]
{\bf Definition.}
In the situation above, assume that $T_p(f)$ is $a$-hyperbolic.
An immersed submanifold $N\sub M_0$
is called a \emph{local $a$-unstable manifold} around~$p$ with respect to~$f$
if
\begin{itemize}
\item[(a)]
$p\in N$;
\item[(b)]
$N$ is tangent at $p$ to
the $a$-unstable subspace $T_p(M)_{a,\obu}$ with respect to $T_p(f)$, i.e.,
$T_p(N)=T_p(M)_{a,\obu}$;
\item[(c)]
There exists an open neighbourhood $U$ of $p$ in $N$
such that $f(U)\sub N$ and $f|_U\colon U\to N$ is analytic.
\end{itemize}
%
%
Combining Theorem~A with
\cite[Theorems 1.9, 1.10, 6.6 and 8.3]{EXP}
(which contain further information),
we obtain in the finite-dimensional
case:\\[3.5mm]
{\bf Local Invariant Manifold Theorem.}
\emph{Let $M$ be a finite-dimensional\linebreak
analytic manifold
over a complete ultrametric field $\K$,
$M_0\sub M$ be an open\linebreak
subset,
$f\colon M_0\to M$ be an analytic map
and $p\in M_0$ a point fixed by $f$.\linebreak
If $a\in \;]0,1]$, then} (a)--(c) \emph{hold}:
\begin{itemize}
\item[(a)]
\emph{There exists an $a$-centre-stable manifold
$N$ around $p$ with
respect to $f$,
such that $N$ is a submanifold of $M$};
\item[(b)]
\emph{If $\alpha:=T_p(f)$ is an automorphism,
then there exists an $a$-centre manifold $N$
around $p$ with respect to $f$ which is a submanifold of $M$,
such that $f(N)=N$};
\item[(c)]
\emph{If $a\not\in R(\alpha)$,
then there exists a local $a$-stable manifold $N$ around $p$
with respect to $f$,
which is a submanifold of $M$}.
\end{itemize}
\emph{For $a\geq 1$, we have}:
\begin{itemize}
\item[(d)]
\emph{If $a\not\in R(\alpha)$,
then there exists a local $a$-unstable manifold $N$ around $p$
with respect to $f$,
which is a submanifold of $M$}.
\end{itemize}
\emph{In all of} (a)--(d), \emph{the germ of $N$
at $p$} (\emph{as an analytic manifold})
\emph{is uniquely determined.
Moreover, there is a basis of open neighbourhoods
$N'$ of $p$ in $N$ such that $N'$ has the property of $N$
described in} (a)--(d), \emph{respectively.}\\[3.5mm]
If $\alpha:=T_p(f)\colon T_p(M)\to T_p(M)$ is an automorphism
in the preceding situation,
then properties of the spectrum of~$\alpha$
and properties of the fixed point~$p$ of~$f$
can be related.
The next theorem collects results of this type
from Propositions \ref{sin}, \ref{typeR} and
\ref{unicon}.
We say that a fixed point $p\in M_0$ of $f\colon M_0\to M$
is \emph{uniformly attractive} if
each neighbourhood
of $p$ in $M_0$ contains a neighbourhood~$Q$ of~$p$ in~$M_0$
such that $f(Q)\sub Q$ and
\[
\lim_{n\to\infty}f^n(x)=p\quad\mbox{for all $\,x\in Q$}
\]
(cf.\ Definition~\ref{typesfp}).\\[3.5mm]
{\bf Theorem B.}
\emph{Let $M$ be a finite-dimensional
analytic manifold
over a complete ultrametric field $\K$,
$M_0\sub M$ be an open
subset,
$f\colon M_0\to M$ be an analytic map
and $p\in M_0$ a fixed point of $f$ such that $\alpha:=T_p(f)$
is an automorphism.
Then} (a)--(c) \emph{hold}:
\begin{itemize}
\item[\rm(a)]
\emph{$R(\alpha)\sub\;]0,1]$
if and only if each neighbourhood
$P$ of $p$ in $M_0$ contains a neighbourhood $Q$
of $p$ such that $f(Q)\sub Q$};
\item[\rm(b)]
\emph{$R(\alpha)\sub \{1\}$ if and only if each
each
neighbourhood $P$ of $p$ in $M_0$ contains a neighbourhood $Q$
of $p$ such that $f(Q)=Q$};
\item[\rm(c)]
\emph{$R(\alpha)\sub\;]0,1[$
if and only if $p$ is a uniformly attractive fixed point
of~$f$.}
\end{itemize}
In the $1$-dimensional case,
fixed (and periodic) points
were already classified into attractive, repelling and indifferent ones
in~\cite{Khr}.
Results concerning attractive and repelling fixed points,
as well as Siegel disks were also obtained in
\cite{Aga},
which amount to the sufficiency (but not the necessity)
of the spectral condition in (b) and (c)
of Theorem~B.\\[2.5mm]
It is useful to have conditions
ensuring that the (global) stable manifold $W^\obs$
is not only an immersed submanifold,
but a submanifold.
In view of Theorem~A,
our Proposition~\ref{Bangd}
below subsumes the following:\\[3.5mm]
{\bf Theorem C.}
\emph{Let $M$ be a finite-dimensional
analytic manifold over a complete ultrametric field.
Let $p\in M$ be a fixed point of
an analytic diffeomorphism
$f\colon M\to M$,
and $\alpha:=T_p(f)$.
If $R(\alpha)\sub \,]0,1]$,
then $W_a^\obs (f,p)$ is a submanifold
of~$M$, for each $a\in \;]0,1]$
such that $T_p(f)$ is $a$-hyperbolic.}\\[3.5mm]
If $\beta \colon G\to G$
is an automorphism of a finite-dimensional
analytic Lie group
$G$ over a complete ultrametric field,
then the neutral element $e\in G$ is a fixed
point for $\beta$,
but we cannot expect in general
that $T_e(\beta)$ is hyperbolic.
Nonetheless, it is always possible
to turn the stable set
\[
U_\beta:=W^{\obs}(\beta,e):=\{x\in G\colon \lim_{n\to\infty}\beta^n(x)=e\}
\]
(the so-called \emph{contraction group}) into
a manifold:\\[3.5mm]
{\bf Theorem D.}
\emph{If $\beta\colon G\to G$
is an automorphism of
a finite-dimensional
analytic Lie group $G$ over a complete ultrametric field,
then there is a unique
immersed submanifold structure on
$U_\beta =W^\obs(\beta,e)$
such that
conditions} (a)--(c)
\emph{of the Ultrametric Stable Manifold Theorem}
(\emph{with $\beta$ in place of $f$})
\emph{are satisfied.
This immersed submanifold structure
makes $U_\beta$
an immersed Lie subgroup
of~$G$.}\\[3.5mm]
To explain the motivation for the current article,
and to show the utility of its results,
we now briefly describe three Lie-theoretic applications which are
only available through the use of invariant manifolds.\\[3mm]
{\bf Applications in Lie theory.}
Let $G$
be an analytic finite-dimensional Lie group over a local field $\K$
and $\beta\colon G\to G$ be
an analytic automorphism.
The \emph{Levi factor} of $\beta$ is the subgroup
\[
M_\beta:=\{x\in G\colon \mbox{$\beta^\Z(x)$ is relatively compact in $G$} \},
\]
where $\beta^\Z(x):=\{\beta^n(x)\colon n\in \Z\}$ (see \cite{BaW}).
Using invariant manifolds,
one can prove the following results
in arbitrary characteristic
(the $p$-adic case of which is due to J.\,S.\,P. Wang \cite{Wan}):
\begin{itemize}
\item[(a)]
\emph{The group $U_\beta$ is always nilpotent}
(see \cite[Theorem B]{MaZ}).
\item[(b)]
\emph{If $U_\beta$ is closed, then $U_\beta$, $U_{\beta^{-1}}$
and $M_\beta$ are Lie subgroups of $G$.
Moreover, $U_\beta M_\beta U_{\beta^{-1}}$ is an open subset
of $G$ and the ``product map''}
\[
\pi\colon U_\beta\times M_\beta \times U_{\beta^{-1}}\to
U_\beta M_\beta U_{\beta^{-1}}\,,\quad (x,y,z)\mto xyz
\]
\emph{is an analytic diffeomorphism}
(see \cite{SPO}).
\end{itemize}
In fact, the $a_j$-stable manifolds $G_j:=W_{a_j}^\obs(\beta,e)$
provide a central series
$\{1\}=G_1\sub G_2\sub\cdots \sub G_n=G$
of Lie subgroups of $G$,
for suitable real numbers
$0<a_1<\cdots < a_n<1$ (see \cite{MaZ}).
And to get (b),
one heavily uses
the (stable) manifold structures on
$U_\beta=W^\obs(\beta,e)$ and
$U_{\beta^{-1}}=W^\obs(\beta^{-1},e)$ discussed here,
and the fact that $M_\beta$ contains a centre manifold
for $\beta$ around~$e$ (see \cite{SPO};
the result was also
announced with a sketch of proof
in \cite[Theorem 9.1]{SUR}).
\begin{itemize}
\item[(c)] Using (b) as a tool,
it is also possible to calculate the
``scale'' $s(\beta)$ (introduced in \cite{Wi1}, \cite{Wi2})\footnote{The scale can be defined
as the minimum index
$s(\beta):=\min_V [V:V\cap \beta^{-1}(V)]$,
for $V$ ranging through the set of all
compact, open subgroups of~$G$.}
if $U_\beta$ is closed,
in terms of the eigenvalues
of the tangent map $L(\beta):=T_e(\beta)$
(see \cite{SPO}; cf.\ \cite[Theorem 9.3]{SUR}
for a more detailed announcement with a sketch of proof).
Previously, this was only possible
in the $p$-adic case
(see \cite{SCA}; cf.\ also \cite{BaW}
for the scale of inner automorphisms
of reductive algebraic groups).
\end{itemize}
\emph{Structure of the article.}
We first provide notation, basic facts and further
definitions of invariant vector subspaces
in a preparatory section (Section~1).
Sections 2--6 are devoted to the proofs of Theorems
A--D, and related results.
\section{Preliminaries and notation}\label{secprepa}
In this section, we fix notation
and recall some basic facts.
We also define (and briefly discuss)
centre subspaces and centre-stable
subspaces.\\[2.5mm]
In this article,
$\N:=\{1,2,\ldots\}$ and $\N_0:=\N\cup\{0\}$. We write
$\Z$
for the integers
and $\R$ for the field of real numbers.
If $f\colon M\to M$ and $n\in \N$,
we write $f^n:=f\circ\, \cdots\, \circ f$
for the $n$-fold composition,
and $f^0:=\id_M$. If $f$ is invertible,
we define $f^{-n}:=(f^{-1})^n$.\\[2.5mm]
Recall that an \emph{ultrametric field}
is a field $\K$, together with
an absolute value $|.|\colon \K\to [0,\infty[$
which satisfies the ultrametric inequality.
We shall always assume that the metric $d\colon \K \times \K \to [0,\infty[$,
$d(x,y):=|x-y|$, defines a non-discrete
topology on $\K$. If the metric space
$(\K,d)$ is complete,
then the ultrametric field $(\K,d)$ is called \emph{complete}.
A totally disconnected, locally compact, non-discrete
topological field is called a \emph{local field}.
Any such admits an ultrametric absolute value
making it a complete ultrametric field \cite{Wei}.
See, e.g., \cite{Sch} for background concerning
complete ultrametric fields.\\[2.5mm]
An \emph{ultrametric Banach space} over an ultrametric field
$\K$ is a complete normed space $(E,\|.\|)$ over $\K$
whose norm $\|.\| \colon E\to [0,\infty[$ satisfies the \emph{ultrametric inequality},
$\|x+y\|\leq \max\{\|x\|, \|y\|\}$ for all $x,y\in E$
(cf.\ \cite{Roo}).
The ultrametric inequality entails that
%
\begin{equation}\label{domi}
\|x+ y\|= \|x\|\quad\mbox{for all $x,y\in E$ such that $\|y\|<\|x\|$.}
\end{equation}
Given $x\in E$ and $r\in \;]0,\infty]$, we set
$B^E_r(x):=\{y\in E\colon \|y-x\|<r\}$.\\[2.5mm]
If $A\colon E\to F$ is a continuous linear map
between ultrametric Banach spaces $(E,\|.\|_E)$ and $(F,\|.\|_F)$,
we write $\|A\|:=\sup\{\|Ax\|_F/\|x\|_E\colon 0\not= x \in E\}$
for its operator norm.
The following observation is immediate.
\begin{numba}\label{expafac}
If $(E,\|.\|)$ is an ultrametric Banach space over
$\K$ and $A\colon E\to E$ an invertible
continuous linear map, then
$\frac{1}{\|A^{-1}\|}$
can be interpreted as an expansion factor,
in the sense that $\|A y\|\geq \frac{1}{\|A^{-1}\|}\|y\|$ for all $y\in E$
(as in the familiar case of real Banach spaces).
\end{numba}
We refer to \cite{Bo1} for the concept of an analytic map
$f\colon U\to F$, where $(E,\|.\|_E)$ and $(F,\|.\|_F)$ are ultrametric Banach
spaces and $U$ is an open subset of~$E$;
compare~\cite{Ser} if $E$ and $F$ have finite dimension.
Thus, in the terminology of Non-Archimedean Geometry,
the mappings we consider are \emph{locally}
analytic\linebreak
maps.
If $f$ is as before and $x\in U$, we write $f'(x)\colon E\to F$
for the total differential of $f$ at $x$.
We shall use that $f$ is \emph{strictly}
differentiable at~$x$~(see~\cite{Bo1}):
%
%
%
\begin{numba}\label{remstrict}
If $f\colon E\supseteq U\to F$ is analytic and $x\in U$,
write
\begin{equation}\label{raute}
f(y)=f(x)+f'(x).(y-x)+R(y)\quad\mbox{for $\, y\in U$.}
\end{equation}
Then $R|_{B^E_\ve(x)}$ is Lipschitz for small $\ve>0$
in the sense that
\[
\Lip(R|_{B^E_\ve(x)}):=\sup\left\{\frac{\|R(z)-R(y)\|_F}{\|z-y\|_E}
\colon y\not= z\in B^E_\ve(x)\right\}<\infty,
\]
and
\[
\lim_{\ve\to0}\Lip(R|_{B^E_\ve(x)})=0.
\]
If $E=F$ and $f'(x)$ is an automorphism,
then
\[
\Lip(R|_{B^E_\ve(x)})<\frac{1}{\|f'(x)^{-1}\|}
\]
for $\ve>0$ small enough.
Hence, by (\ref{domi}) and (\ref{raute}),
for all $y,z\in B^E_\ve(x)$ we have
\begin{equation}\label{flower}
\|f(z)-f(y)\|=\|f'(x)(z-y)+R(z)-R(y)\|=\|f'(x).(z-y)\|.
\end{equation}
\end{numba}
An \emph{analytic manifold} modelled on an
ultrametric Banach space $E$ over a complete ultrametric field $\K$ is defined
as usual (as a Hausdorff topological
space~$M$, together with a (maximal) set $\cA$ of homeomorphisms (``charts'')
$\phi\colon U_\phi\to V_\phi$
from open subsets of~$M$
onto open subsets of~$E$,
such that $M=\bigcup_{\phi\in \cA}U_\phi$
and the mappings
$\phi\circ \psi^{-1}$
are analytic for all $\phi,\psi\in \cA$).
Also the tangent space $T_pM$ of $M$ at $p\in M$,
the tangent bundle $TM$,
analytic maps $f\colon M\to N$ between analytic manifolds,
and the tangent maps $T_pf\colon T_pM\to T_{f(p)}N$
as well as $Tf\colon TM\to TN$
can be defined as usual (cf.\ \cite{Bo1}).
If $f\colon M\to V$ is an analytic map to an open subset
$V$ of an ultrametric Banach space $F$,
then we identify $TV$ with $V\times F$ in the natural way
and let $df\colon TM\to F$ be the second component of the map
$Tf\colon M\to V\times F$.
An \emph{analytic Lie group}~$G$ over $\K$
is a group, equipped with an analytic manifold structure
modelled on an ultrametric
Banach space over $\K$,
such that the group inversion and group
multiplication are analytic (cf.\ \cite{Bo2}).
As usual, we write $L(G):=T_e(G)$
and $L(\beta):=T_e(\beta)$, if $\beta\colon G\to H$
is an analytic homomorphism between analytic Lie groups.
Let $M$ be an analytic manifold modelled on an ultrametric Banach space~$E$.
A subset $N\sub M$ is called a \emph{submanifold} of
$M$ if there exists a complemented vector subspace $F$ of the modelling space
of $M$ such that
each point $p\in N$ is contained in
the domain~$U$ of some chart $\phi\colon U\to V$ of $M$
such that $\phi(N\cap U)=F\cap V$.
By contrast, an analytic manifold $N$ is called an \emph{immersed submanifold} of $M$
if $N\sub M$ as a set and the inclusion map $\iota \colon N\to M$
is an immersion.
Subgroups of Lie groups with analogous properties are called
\emph{Lie subgroups}
and \emph{immersed Lie subgroups}, respectively.
If we call a mapping $f$ an analytic diffeomorphism between two manifolds
(or an analytic automorphism of a Lie group),
then also the inverse map $f^{-1}$ is assumed analytic.\\[2.5mm]
Let us now complete the definitions
of invariant vector subspaces from the Introduction.
In the remainder of this section,
let $E$ be an ultrametric Banach space over $\K$.
Let $\alpha\colon E\to E$
be a continuous linear map,
and $a\in \;]0,\infty[$.
\begin{rem}\label{centunique}
We mention that the spaces
$E_{a,\obs}$ and $E_{a,\obu}$
in the definition of $a$-hyperbolicity
stated in the Introduction
are uniquely determined,
in the case of an endomorphism
$\alpha\colon E\to E$
of a finite-dimensional
$\K$-vector space~$E$.
See \cite[Remark~6.4]{EXP}
for the assertion if
$\alpha$ is an automorphism.
In the general case,
the argument in the cited remark still
provides uniqueness of $E_{a,\obs}$.
Let us write $E^+:=\bigcap_{k\in \N}\alpha^k(E)$
for the Fitting one component of $E$
(see, e.g., \cite[Lemma 5.3.11]{HaN}).
Then $\alpha$ restricts to an automorphism $\beta$
of $E^+$. Now
$E^+=(E_{a,\obs})^+ \oplus E_{a,\obu}$
is a decomposition for the $a$-hyperbolic automorphism~$\beta$
and thus also $E_{a,\obu}$ is unique.
\end{rem}
%
%
%
\begin{defn}\label{defcsub}
An $\alpha$-invariant vector subspace
$E_{a,\obcs} \sub E$ is called an \emph{$a$-centre-stable
subspace} with respect to $\alpha$
if there exists an $\alpha$-invariant
vector subspace $E_{a,\obu}$ of $E$ such that
$E=E_{a,\obcs} \oplus E_{a,\obu}$ and
$\alpha_2:= \alpha|_{E_{a,\obu}}\colon E_{a,\obu}\to E_{a,\obu}$
is invertible,
and there exists an ultrametric norm $\|.\|$ on $E$ defining its topology,
with the following properties:
\begin{itemize}
\item[(a)]
$\|x+y\|=\max\{\|x\|,\|y\|\}$
for all $x\in E_{a,\obcs}$, $y\in E_{a,\obu}$; and
\item[(b)]
$\|\alpha_1\|\leq a$ and $\frac{1}{\|\alpha_2^{-1}\|}>a$
holds for the operator norms with respect to $\|.\|$,
where $\alpha_1:=\alpha|_{E_{a,\obcs}}$.
\end{itemize}
Then $E_{a,\obcs}$ is uniquely determined
and if $\alpha$ is invertible, then $E_{a,\obu}$ is unique
(see \cite[Remark 3.3]{EXP}).
Arguing as in Remark~\ref{centunique},
we see that
$E_{a,\obu}$ is also unique if $E$
is finite-dimensional.
\end{defn}
%
\begin{defn}\label{defacentre}
We say that an $\alpha$-invariant vector subspace
$E_{a,\obc} \sub E$ is an \emph{$a$-centre
subspace} with respect to $\alpha$
if there exist $\alpha$-invariant
vector subspaces $E_{a,\obs}$ and $E_{a,\obu}$ of $E$ such that
$E=E_{a,\obs}\oplus E_{a,\obc} \oplus E_{a,\obu}$,
and an ultrametric norm $\|.\|$ on $E$ defining its topology,
with the following properties:
\begin{itemize}
\item[(a)]
$\|x+y+z \|=\max\{\|x\|,\|y\|,\|z\|\}$
for all $x\in E_{a,\obs}$, $y\in E_{a,\obc}$  and $z\in E_{a,\obu}$;
\item[(b)]
$\|\alpha(x)\|=a\|x\|$ for all $x\in E_{a,\obc}$;
\item[(c)]
$\alpha_3:=\alpha|_{E_{a,\obu}}$
is invertible;\footnote{This hypothesis
can be omitted (as it then follows from the others) if $E$ has finite dimension
(since $\ker\alpha\sub E_{a,\obs}$)
or $\alpha$ is an automorphism.}
and
\item[(d)]
$\|\alpha_1\|<a$ and $\frac{1}{\|\alpha_3^{-1}\|}>a$ hold for
the operator norms with respect
to $\|.\|$, where $\alpha_1:=\alpha|_{E_{a,s}}$.
\end{itemize}
If $\alpha$ is an automorphism,
then
$E_{a,\obs}$, $E_{a,\obc}$ and
$E_{a,\obu}$
are uniquely determined (see \cite[Remark 4.3]{EXP}).
If $E$ is finite-dimensional, then $E_{a,\obs}$ is unique by its description
in \cite[Remark 4.3]{EXP},
and hence also $E_{a,\obc}$ and $E_{a,\obu}$
are unique by the argument from Remark~\ref{centunique}.
$E_{a,\obs}$ and $E_{a,\obu}$ are
called the \emph{$a$-stable}
and \emph{$a$-unstable} subspaces of $E$ with respect to~$\alpha$, respectively.
If $a=1$, we simply speak of stable, centre and unstable
subspaces, and write~$E_\obs$, $E_\obc$ and
$E_{\obu}$ instead of
$E_{1,\obs}$, $E_{1,\obc}$ and~$E_{1,\obu}$.
\end{defn}
\section{Spectral interpretation of hyperbolicity}\label{secfin}
%
In this section, we consider
the special case where~$\alpha$ is an automorphism
of a \emph{finite-dimensional} vector space
over a complete ultrametric field
$(\K,|.|)$.
We shall interpret $a$-hyperbolicity
as the absence of eigenvalues
of absolute value $a$ (in an
algebraic closure of $\K$).
Moreover, we shall see that an $a$-centre subspace
and an $a$-centre-stable subspace always exist.
%
%
\begin{numba}\label{findiset}
Let $(\K,|.|)$ be a complete ultrametric
field, $E$ be a finite-dimensional $\K$-vector
space, and $\alpha\colon E\to E$
be a linear map.
We define $\wb{\K}$,
the extension $|.|$ and $R(\alpha)$
as in the Introduction,
using
the $\wb{\K}$-linear self-map
$\alpha_{\wb{\K}}:=\alpha\otimes \id_{\wb{\K}}$
of the $\wb{\K}$-vector space
$E_{\wb{\K}}:=E\tensor_\K \wb{\K}$
obtained from~$E$ by extension of scalars.
For each $\lambda\in \wb{\K}$,
we let
\[
(E_{\wb{\K}})_{(\lambda)}\, :=\, \{ x\in E_{\wb{\K}}\!:
(\alpha_{\wb{\K}}-\lambda)^dx=0\}
\]
be the generalized eigenspace of $\alpha_{\wb{\K}}$
in $E_{\wb{\K}}$ corresponding to~$\lambda$
(where $d$ is the dimension of the $\K$-vector space~$E$).
Given $\rho\in [0,\infty[$,
we define
\begin{equation}\label{dfspacerho}
(E_{\wb{\K}})_\rho\; :=\;
\bigoplus_{|\lambda|=\rho} (E_{\wb{\K}})_{(\lambda)}\,\sub\, E_{\wb{\K}} \, ,\vspace{-.7mm}
\end{equation}
where the sum is taken over all
$\lambda\in \wb{\K}$
such that $|\lambda|=\rho$.
As usual, we identify $E$ with $E\tensor 1\sub E_{\wb{\K}}$.
\end{numba}
The following fact (cf.\ (1.0) on p.\,81 in \cite[Chapter~II]{Mar})
is important:\footnote{In \cite[p.\,81]{Mar},
$\K$ is a local field,
but the proof works also for complete ultrametric fields.}
\begin{la}
For each $\rho\in R(\alpha)$,
the vector
subspace $(E_{\wb{\K}})_\rho$ of $E_{\wb{\K}}$ is defined
over~$\K$, i.e.,
$(E_{\wb{\K}})_\rho= (E_\rho)_{\wb{\K}}$
with $E_\rho:=(E_{\wb{\K}})_\rho\cap E$.
Thus
\begin{equation}\label{isdsum}
E\; =\; \bigoplus_{\rho\in R(\alpha)} E_\rho\,,
\end{equation}
and each $E_\rho$
is an $\alpha$-invariant vector subspace of~$E$.\,\Punkt
\end{la}
It is essential for us that certain well-behaved norms
exist on~$E$ (as in~\ref{findiset}).
%
%
\begin{defn}\label{defnadpt}
A norm $\|.\|$
on $E$ is \emph{adapted to $\alpha$} if the following holds:
\begin{itemize}
\item[(a)]
$\|.\|$ is ultrametric;
\item[(b)]
$\big\|\sum_{\rho\in R(\alpha)} x_\rho\big\|=\max\{\|x_\rho\|
\colon \rho\in R(\alpha)\}$\vspace{.5mm}
for each $(x_\rho)_{\rho\in R(\alpha)}
\in \prod_{\rho\in R(\alpha)} E_\rho$; and
\item[(c)]
$\|\alpha (x)\|=\rho\|x\|$ for each $0\not= \rho\in R(\alpha)$
and $x\in E_\rho$.
\end{itemize}
\end{defn}
%
%
\begin{prop}\label{propadapt}
Let $E$ be a finite-dimensional vector space over
a complete ultrametric field $(\K,|.|)$ and
$\alpha\colon E\to E$ be a linear map.
Let $\ve>0$
and $E_0:=\{x\in E\colon (\exists n\in \N)\;\alpha^n(x)=0\}$.
Then $E$ admits a norm $\|.\|$ adapted to~$\alpha$,
such that $\alpha|_{E_0}$ has operator norm
$<\ve$ with respect to $\|.\|$.
\end{prop}
The proof uses the following lemma:
%
\begin{la}\label{normrho}
For each $\rho\in R(\alpha)\setminus\{0\}$,
there exists an ultrametric norm
$\|.\|_\rho$ on $E_\rho$ such that $\|\alpha(x)\|_\rho=\rho\|x\|_\rho$
for each $x\in E_\rho$.\,\Punkt
\end{la}
\begin{proof}
If $\alpha$ is an automorphism, then the assertion holds by
\cite[Lemma~4.4]{SUR}.
The general case follows if we replace $\alpha$
by the map $\alpha|_{E_\rho}\colon E_\rho\to E_\rho$,
which is an automorphism as $\ker(\alpha)\sub E_0$
and thus $E_\rho\cap \ker(\alpha)=\{0\}$.
\end{proof}
The next lemma takes care of the case $\rho=0$.
%
\begin{la}\label{rhozero}
Let $E$ be a finite-dimensional vector space over a complete
ultrametric field $(\K,|.|)$
and $\alpha\colon E\to E$ be a nilpotent
linear map. Let $\ve>0$.
Then there exists an ultrametric norm $\|.\|$ on~$E$
with respect to which $\alpha$ has operator norm $<\ve$.
\end{la}
\begin{proof}
Assume first that there exists a basis $v_1,\ldots, v_m$
of $E$ with respect to which $\alpha$
has Jordan normal form with a single Jordan block,
i.e., $\alpha(v_1)=0$ and $\alpha(v_k)=v_{k-1}$
for $k\in \{2,\ldots, m\}$.
The case $E=\{0\}$ being trivial,
we may assume that $m\geq 1$.
Choose $\lambda\in \K$ such that $0<|\lambda|<\ve$
and define $w_k:=\lambda^k v_k$
for $k\in \{1,\ldots, m\}$.
Then $\alpha(w_k)=\lambda^k v_{k-1}=\lambda w_{k-1}$
for $k\in\{2,\ldots, m\}$ and $\alpha(w_1)=0$,
entailing that $\alpha$ has operator norm
$<\ve$ with respect to the maximum norm $\|.\|$
on $E$ with respect to the basis $w_1,\ldots, w_m$,
\[
\left\|\sum_{k=1}^m t_kw_k\right\|\; :=\;
\max\{|t_k|\colon k=1,\ldots, m\}\quad \mbox{for $\, t_1,\ldots, t_m\in \K$.}
\]
In the general case, we write
$E$ as a direct sum $\bigoplus_{j=1}^nE_j$
of $\alpha$-invariant
vector subspaces $E_j\sub E$
such that the Jordan decomposition of $\alpha|_{E_j}$
has a single Jordan block.
For each $j$,
there exists
an ultrametric norm $\|.\|_j$ on $E_j$
with respect to which $\alpha|_{E_j}$ has operator
norm $<\ve$,
by the above special case.
Then $\alpha$ has operator norm $<\ve$
with respect to the ultrametric norm $\|.\|$ on~$E$
given by $\|v_1+\cdots +v_n\|:=\max\{\|v_j\|_j\colon j=1,\ldots, n\}$
for $v_j\in E_j$.
\end{proof}
{\bf Proof of Proposition~\ref{propadapt}.}
For each $\rho\in R(\alpha)\setminus\{0\}$, we choose a norm $\|.\|_\rho$
on $E_\rho$ as described in Lemma~\ref{normrho}.
Lemma~\ref{rhozero}
provides an ultrametric norm $\|.\|_0$
on $E_0$, with respect to which
$\alpha|_{E_0}$ has operator norm $<\ve$.
Then
\[
\Big\| \sum_{\rho\in R(\alpha)} x_\rho\Big\| \; :=
\; \max\, \big\{ \,\|x_\rho\|_\rho\colon \rho\in R(\alpha)\,\big\}\quad
\mbox{for $\,(x_\rho)_{\rho\in R(\alpha)} \in \prod_{\rho\in R(\alpha)}
E_\rho$}
\]
defines a norm $\|.\|\colon E\to[0,\infty[$
which, by construction, is adapted to~$\alpha$
and with respect to which
$\alpha|_{E_0}$ has operator norm $<\ve$.\,\vspace{3mm}\Punkt

\noindent
We are now ready to prove Theorem~A from the Introduction.\\[2.5mm]
{\bf Proof of Theorem A.}
By Proposition~\ref{propadapt},
there exists an ultrametric norm $\|.\|\wt{\;}$
on $E$ which is adapted to~$\alpha$,
and with respect to which
$\alpha|_{E_0}$ has operator norm $<a$.

\emph{Centre-stable subspaces.} The conditions
from Definition~\ref{defcsub}
are satisfied with $\|.\|:=\|.\|\wt{\;}$
and
%
\begin{equation}\label{sodecacs}
E_{a,\obcs}:=\bigoplus_{\rho\leq a} E_\rho\quad\mbox{ and }\quad
E_{a,\obu}:= \bigoplus_{\rho >a} E_\rho\,.
\end{equation}

\emph{Centre subspaces.}
The conditions
of Definition~\ref{defacentre}
are satisfied with $\|.\|:=\|.\|\wt{\;}$ and
%
\begin{equation}\label{sodecac}
E_{a,\obs}:=\bigoplus_{\rho< a} E_\rho,\quad
E_{a,\obc}:=E_a,\quad
\quad\mbox{ and }\quad
E_{a,\obu}:= \bigoplus_{\rho >a} E_\rho\,.
\end{equation}

\emph{Hyperbolicity.}
If $a\not\in R(\alpha)$,
then the conditions
from the definition of $a$-hyperbolicity
(stated in the Introduction)
are satisfied with $\|.\|:=\|.\|\wt{\;}$,
%
\begin{equation}\label{soahyp}
E_{a,\obs}:=\bigoplus_{\rho< a} E_\rho \quad
\quad\mbox{ and }\quad
E_{a,\obu}:= \bigoplus_{\rho >a} E_\rho\,.
\end{equation}
If $a\in R(\alpha)$,
then $\alpha$ cannot be $a$-hyperbolic.
In fact,
if $\alpha$ was $a$-hyperbolic,
we obtain a norm $\|.\|$
and a splitting $E=E_{a,\obs}\oplus E_{a,\obu}$
as in the cited definition.
Define $\alpha_1:=\alpha|_{E_{a,\obs}}$ and
$\alpha_2:=\alpha|_{E_{a,\obu}}$.
Because the norms $\|.\|$ and $\|.\|\wt{\;}$
are equivalent, there exists
$C>0$ such that $C^{-1}\|.\|\leq \|.\|\wt{\;}\leq C\|.\|$.
Let $0\not=v\in E_a$.
Write $v=x+y$ with $x\in E_{a,\obs}$ and $y\in E_{a,\obu}$.
If $y\not=0$, then
\[
\|v\|\wt{\;}=a^{-n}\|\alpha^n(v)\|\wt{\;}
\geq
a^{-n} C^{-1}\|\alpha^n(v)\|
\geq C^{-1} \left(\frac{1}{a \|\alpha_2^{-1}\|}\right)^n \|y\|
\]
for all $n\in \N$, which is absurd because
$\frac{1}{a\|\alpha_2^{-1}\|}>1$.
Hence $y=0$ and thus $x=v\not=0$.
But then
\[
\|v\|\wt{\;}=a^{-n}\|\alpha^n(v)\|\wt{\;}
\leq a^{-n} C \|\alpha^n(v)\|
\leq C \left(\frac{\|\alpha_1\|}{a}\right)^n\|v\|
\quad
\mbox{for all $\, n\in \N$.}
\]
Since $\frac{\|\alpha_1\|}{a}<1$,
this is absurd. Thus $\alpha$
cannot be $a$-hyperbolic.\,\vspace{1mm}\Punkt
\section{Behaviour close to a fixed point}\label{seccon}
%
We now relate
the behaviour of a dynamical system $(M,f)$
around a fixed point~$p$
and properties of the linear map $T_p(f)$.
%
%
\begin{numba}\label{situnow}
Let $M$ be an analytic manifold modelled on an ultrametric Banach space
over a complete ultrametric field $(\K,|.|)$.
Let $f\colon M_0\to M$ be an analytic
mapping on an open subset $M_0\sub M$ and $p\in M_0$ be a fixed
point of $f$, such that $T_p(f)\colon T_p(M)\to T_p(M)$
is an automorphism.
\end{numba}
%
%
%
%
\begin{prop}\label{sin}
In \,{\rm\ref{situnow}},
the following conditions are equivalent:
\begin{itemize}
\item[\rm(a)]
$T_p(M)$ admits a centre-stable
subspace with respect to $T_p(f)$,
and
each neighbourhood $P$ of $p$ in $M_0$ contains a neighbourhood $Q$
of $p$ such that $f(Q)\sub Q$.
\item[\rm(b)]
There exists a norm $\|.\|$ on $T_p(M)$ defining its topology,
such that $\|T_p(f)\|\leq 1$ holds for the corresponding operator norm.
\end{itemize}
If, moreover, $M$ is a finite-dimensional
manifold, then {\rm(a)} and {\rm(b)}
are also equivalent to the following condition:
\begin{itemize}
\item[\rm(c)]
Each eigenvalue $\lambda$ of $T_p(f)\otimes_\K \id_{\wb{\K}}$
in an algebraic closure $\wb{\K}$ of $\K$ has absolute value
$|\lambda|\leq 1$.
\end{itemize}
\end{prop}
\begin{proof}
(b) means that $E:=T_p(M)$ coincides with
its centre-stable subspace with respect to $\alpha:=T_p(f)$.
If $E$ is finite-dimensional,
this property is equivalent to $R(\alpha)\sub \;]0,1]$
and hence to~(c),
by~(\ref{sodecacs}) (using that $E_{\obcs}$
is unique).
If~(b) holds, then (a) follows with
\cite[Theorem~1,9\,(c)]{EXP}.\footnote{If $E$ is finite-dimensional, this corresponds
to the conclusions concerning centre-stable
manifolds in the Local Invariant Manifold Theorem stated above.}

(a)$\impl$(b): If (a) holds,
then $E$ admits a decomposition
$E=E_{1,\obcs}\oplus E_{1,\obu}$
and a norm $\|.\|$, as described in Definition~\ref{defcsub}
(with $a=1$). After shrinking $M_0$, we may
assume that $M_1:=f(M_0)$ is open in $M$ and $f \colon M_0\to M_1$
is a diffeomorphism (by the Inverse Function Theorem).

If $E_{1,\obu}\not=\{0\}$,
we let $P\sub M_0\cap M_1$
be an open neighbourhood of $p$ such that $f(P)\sub P$,
and consider the map $g:=f^{-1}\colon M_1\to M$.
Then $E_{1,\obu}$ is the stable subspace of~$E$
with respect to $T_p(g)=\alpha^{-1}$.
Pick $b\in \;] \|\alpha^{-1}|_{E_{1,\obu}}\| ,1[$.
Then $\alpha^{-1}$ is $b$-hyperbolic, and
\[
E_{b,\obs}=E_{1,\obu} \quad \mbox{ as well as  }\quad E_{b,\obu}=E_{1,\obcs}
\]
(with respect to the automorphisms $\alpha^{-1}$ and $\alpha$
on the left and right of the equality signs, respectively).
By \cite[Theorem~6.6]{EXP}
(applied to $g|_P\colon P\to M$),
there exists a local $b$-stable
manifold $N\sub P$ with respect to~$g$,
such that $g^n(x)\to p$ as $n\to\infty$,
for all $x\in N$.
Since $N$ is tangent to
$E_{1,\obu}\not=\{0\}$,
we have $N\not=\{p\}$
and thus find a point $x\in N\setminus\{p\}$.
By hypothesis~(a), there is
an open $p$-neighbourhood $Q\sub P\setminus \{x\}$
with $f(Q)\sub Q$.
Since $g^n(x)\to p$,
there exists $m\in \N$ with
$y:=g^m(x)\in Q$.
Then $x=f^m(y)\in f^m(Q)\sub Q$,
contradicting the choice of~$Q$.
Hence $E_{1,u}=\{0\}$ (and thus (b) holds).
\end{proof}
%
%
%
\begin{prop}\label{typeR}
In \,{\rm\ref{situnow}},
the following conditions are equivalent:
\begin{itemize}
\item[\rm(a)]
$T_p(M)$ admits a centre subspace
with respect to $T_p(f)$, and
each\linebreak
neighbourhood $P$ of $p$ in $M_0$ contains a neighbourhood $Q$
of $p$ such that $f(Q)=Q$.
\item[\rm(b)]
There exists a norm $\|.\|$ on $T_p(M)$ defining its topology,
which makes $T_p(f)$ an isometry.
\end{itemize}
If, moreover, $M$ is a finite-dimensional
manifold, then {\rm(a)} and {\rm(b)}
are also equivalent to the following condition:
\begin{itemize}
\item[\rm(c)]
Each eigenvalue $\lambda$ of $T_p(f)\otimes_\K \id_{\wb{\K}}$
in an algebraic closure $\wb{\K}$ of $\K$ has absolute value
$|\lambda|=1$.
\end{itemize}
\end{prop}
\begin{proof}
(b) means that $E:=T_p(M)$ coincides with
its centre subspace with respect to $\alpha:=T_p(f)$.
If $E$ is finite-dimensional,
this property is equivalent to $R(\alpha)\sub \{1\}$
and hence to~(c),
by (\ref{sodecac}) (using the uniqueness
of $E_{\obc}$).
If~(b) holds, then~(a) follows with
\cite[Theorem~1.10\,(c)]{EXP}.\footnote{If $E$
is finite-dimensional, see also the conclusions concerning centre
manifolds in the Local Invariant Manifold Theorem above.}

(a)$\impl$(b):
After shrinking $M_0$, we may
assume that $M_1:=f(M_0)$ is open in $M$ and $f \colon M_0\to M_1$
is a diffeomorphism.
If (a) holds,
then there is a decomposition
$E=E_{1,\obs}\oplus E_{1,\obc}\oplus E_{1,\obu}$
and a norm $\|.\|$, as in Definition~\ref{defacentre}
(with $a=1$). By ``(a)$\impl$(b)''
in Proposition~\ref{sin},
we have $E_{1,\obu}=\{0\}$.
Applying
Proposition~\ref{sin}
to $g:=f^{-1}\colon M_1\to M$,
we see that also $E_{1,\obs}=\{0\}$
(because this is the unstable subspace of $T_p(M)$
with respect to $T_p(g)=\alpha^{-1}$).
Thus $E=E_{1,\obc}$,
establishing~(b).
\end{proof}
The proofs show that $Q$ can always be chosen
as an \emph{open} subset of $M_0$, in part (a)
of Proposition \ref{sin} and \ref{typeR}.
%
\begin{defn}\label{typesfp}
In the situation of {\rm\ref{situnow}},
we use the following terminology:
\begin{itemize}
\item[(a)]
$p$ is said to be an \emph{attractive} fixed point
of $f$ if $p$  has a neighbourhood
$P\sub M_0$ such that $f^n(x)$ is defined
for all $x\in P$ and $n\in \N$,
and $\lim_{n\to\infty}f^n(x)=p$ for all $x\in P$.
\item[(b)]
We say that $p$ is \emph{uniformly attractive}
if it is attractive and, moreover,
every neighbourhood of~$p$ in~$M_0$
contains a neighbourhood~$Q$ of~$p$ such that
$f(Q)\sub Q$.
\end{itemize}
\end{defn}
%
%
\begin{prop}\label{unicon}
In \,{\rm\ref{situnow}},
the following conditions are equivalent:
\begin{itemize}
\item[\rm(a)]
$T_p(M)$ admits a centre subspace with respect
to $T_p(f)$, and $p$ is\linebreak
uniformly attractive;
\item[\rm(b)]
There exists a norm $\|.\|$ on $T_p(M)$ defining its topology,
such that $\|T_p(f)\|<1$ holds for the corresponding operator norm.
\end{itemize}
If, moreover, $M$ is a finite-dimensional
manifold, then {\rm(a)} and {\rm(b)}
are also equivalent to the following condition:
\begin{itemize}
\item[\rm(c)]
Each eigenvalue $\lambda$ of $T_p(f)\otimes_\K \id_{\wb{\K}}$
in an algebraic closure $\wb{\K}$ of $\K$ has absolute value
$|\lambda| < 1$.
\end{itemize}
\end{prop}
\begin{proof}
(b) means that $E:=T_p(M)$ coincides with
its stable subspace with respect to $\alpha:=T_p(f)$.
If $E$ is finite-dimensional,
this property is equivalent to $R(\alpha)\sub \;]0,1[$
and hence to~(c),
by~(\ref{sodecac}) (using the uniqueness
of $E_{\obs}$). If (a) holds, then also (b),
as shall be verified in Remark~\ref{rmweaker}.

If (b) holds and $P\sub M_0$
is an open neighbourhood of~$p$,
then \cite[Theorem 6.6]{EXP}\footnote{If $E$
is finite-dimensional, see also the conclusions concerning
local stable manifolds in the Local Invariant Manifold Theorem above.}
(applied to $f|_P$ instead of $f$)
provides a local stable manifold $N\sub P$
such that $\lim_{n\to\infty}f^n(x)=p$ for all $x\in N$. 
Because $T_p(N)=E=T_p(M)$,
it follows that~$N$ is open in~$M$.
Since, moreover, $f(N)\sub N$ by definition of $N$,
we have verifed that $p$ is uniformly attractive.
\end{proof}
\begin{rem}\label{rmweaker}
If $p$ is merely attractive (but possibly not
uniformly) and $E:=T_p(M)$
admits a centre subspace with respect to $T_p(f)$,
we can still conclude that $E_{1,\obc}=\{0\}$.\\[2.5mm]
[After shrinking $M_0$, we may assume
that $f$ is injective.
Let $P\sub M_0$ be as in Definition~\ref{typesfp}\,(a).
If $E_{1,\obc}\not=\{0\}$, we let $Q\sub P$
be a centre manifold with respect to~$f$,
such that $f(Q)=Q$ (see \cite[Theorem~1.10\,(c)]{EXP}).
Since $E_{1,\obc}\not=\{0\}$, we must have $Q\not=\{p\}$,
enabling us to pick $x_0\in Q\setminus\{p\}$.
Using
\cite[Theorem~1.10\,(c)]{EXP}
again, we find
a centre manifold $S\sub Q\setminus\{x_0\}$
with respect to~$f$,
such that $f(S)=S$.
Since $f$ is injective, it follows that
$f(Q\setminus S)=Q\setminus S$
and thus $f^n(x_0)\in Q\setminus S$ for all $n\in \N_0$.
As $Q$ is a neighbourhood of~$p$,
we infer $f^n(x_0)\not\to p$ as $n\to\infty$.
Since $x_0\in P$, this contradicts the choice of~$P$.\,]
\end{rem}
\section{When {\boldmath$W_a^\obs(f,p)$} is
not only immersed}\label{notonly}
%
%
In general, $W_a^\obs$ is only an \emph{immersed}
submanifold of $M$,
not a submanifold
(cf.\
\cite[\S7.1]{SUR}
for an easy example).
We now describe
a criterion
(needed in~\cite{MaZ})
which prevents such pathologies.
%
%
\begin{prop}\label{Bangd}
Let $M$ be an analytic manifold modelled on an
ultrametric Banach space over a complete
ultrametric field.
Let $p\in M$ be a fixed point of
an analytic diffeomorphism
$f\colon M\to M$, such that
$E:=T_p(M)$ admits a centre-stable subspace
with respect to $T_p(f)$,
and $E_{1,\obu}=\{0\}$.
Then $W_a^\obs (f,p)$ is a submanifold
of~$M$, for each $a\in \;]0,1]$
such that $T_p(f)$ is $a$-hyperbolic.
\end{prop}
\begin{proof}
Let $W_a^\obs:=W_a^\obs(f,p)$ and
$\Omega \sub W_a^\obs$ be as
in the Ultrametric Stable Manifold Theorem
from the Introduction.
Since $f$ restricts to a diffeomorphism of $W_a^\obs$,
the image $f(\Omega)$ is relatively
open in $\Omega$.
Hence,
there exists an open $p$-neighbourhood $Q\sub M$
such that $\Omega\cap Q\sub f(\Omega)$.
By ``(b)$\impl$(a)'' in Proposition \ref{sin},
we may assume that $f(Q)\sub Q$,
after replacing $Q$ with a smaller
neighbourhood of $p$ if necessary.
We claim that
%
\begin{equation}\label{goclaim}
W_a^\obs \cap Q
=\Omega\cap Q\,.
\end{equation}
If this is true, then
$W_a^\obs \cap Q$
is a submanifold
of $M$, and hence also
\[
f^{-n}(W_a^\obs\cap Q)= f^{-n}(W_a^\obs)\cap f^{-n}(Q)=W_a^\obs\cap f^{-n}(Q)
\]
is a submanifold of~$M$ (as $f^{-n}\colon M\to M$
is a diffeomorphism).
Since $\bigcup_{n\in \N_0}f^{-n}(Q)$ is an open subset
of~$M$ which contains $W_a^\obs$ (exploiting that
$f^n(x)\in Q$ for large~$n$, for each $x\in W_a^\obs$),
we deduce that $W_a^\obs$ is a submanifold of~$M$
(and the submanifold structure
coincides with the given immersed submanifold structure
on $W^\obs_a$, as both structures
coincide on each of the sets $f^{-n}(W^\obs_a\cap Q)$, $n\in\N_0$,
which form an open cover for $W^\obs_a$).\\[2.5mm]
To prove (\ref{goclaim}), suppose
that $x\in W_a^\obs\cap Q$ but $x\not\in \Omega\cap Q$
(and hence $x\not\in \Omega$).
Since $f(Q)\sub Q$, we then have
\[
f^n(x)\in Q\quad\mbox{for all $\,n\in \N_0$.}
\]
By definition of $\Omega$, there exists
$n\in \N_0$ such that $f^n(x)\in \Omega$.
We choose $n$ minimal and note that
$n\geq 1$ as $x\not\in \Omega$
by hypothesis.
Then $f^n(x)\in \Omega\cap Q\sub f(\Omega)$
and hence $f^{n-1}(x)=f^{-1}(f^n(x)) \in f^{-1}(f(\Omega))=\Omega$,
contradicting the minimality of~$n$.
Hence $x$ cannot exist
and thus $W_a^\obs\cap Q\sub \Omega\cap Q$.
The converse inclusion, $\Omega\cap Q\sub W_a^\obs\cap Q$,
being trivial,
(\ref{goclaim}) is proved.
\end{proof}
\section{Dependence of {\boldmath$a$}-stable manifolds
on~{\boldmath$a>0$}}\label{adepend}
%
%
We collect further results
in the finite-dimensional case
required in Section~6 and \cite{MaZ}.
In particular, we
study the dependence of $a$-stable manifolds
on the parameter~$a$.
%
%
\begin{prop}\label{adepprop}
Let $M$ be an analytic manifold modelled on a
finite-\linebreak
dimensional vector space over a complete
ultrametric field $(\K,|.|)$.
Let $p\in M$ be a fixed point of
an analytic diffeomorphism
$f\colon M\to M$.
Abbreviate $\alpha:=T_p(f)$
and define $R(\alpha)$ as in the Introduction.
Then the following holds:
\begin{itemize}
\item[\rm(a)]
If $R(\alpha)\sub \;]0,1]$, then
$W_a^\obs(f,p)$ is a submanifold
of $M$, for each $a\in \;]0,1]\setminus R(\alpha)$.
\item[\rm(b)]
If $0<a<b\leq 1$ and $[a,b]\cap R(\alpha)=\emptyset$,
then $W_a^\obs(f,p)=W_b^\obs(f,p)$.
\item[\rm(c)]
If $a\in \;]0,1]$ and $\,]0,a]\cap R(\alpha)=\emptyset$,
then $W_a^\obs(f,p)=\{p\}$.
\end{itemize}
\end{prop}
\begin{proof}
(a) follows from Proposition~\ref{Bangd}
(using~(\ref{sodecac}) and Theorem~A).

(b)
Define $E:=T_p(M)$.
Let $\|.\|$ be a norm on
$E$ adapted to $\alpha:=T_p(f)$,
and $R(\alpha)$ as well as the subspaces
$E_\rho\sub E$ for $\rho>0$ be as in~\ref{findiset}.
Choose a chart $\kappa\colon P\to U\sub E$ of $M$ around~$p$
such that $\kappa(p)=0$.
Let $Q\sub P$ be an
open neighbourhood of~$p$ such that $f(Q)\sub P$;
after shrinking $Q$, we may assume that $\kappa(Q)=B_r^E(0)$ for some $r>0$.
Then $g:=\kappa\circ f|_Q\circ\kappa^{-1}|_{B_r^E(0)}\colon B^E_r(0)\to E$
expresses $f|_Q$ in the local chart $\kappa$.
By hypothesis on $a$ and $b$, we have
\[
X :=\bigoplus_{\rho<a}E_\rho=\bigoplus_{\rho<b}E_\rho
\quad\mbox{ and }\quad
Y :=\bigoplus_{\rho>a}E_\rho=\bigoplus_{\rho>b}E_\rho\,.
\]
Hence $E_{a,\obs}=E_{b,\obs}=X$
and $E_{a,\obu}=E_{b,\obu}=Y$,
by~(\ref{soahyp}).
Now let $\Omega_a$ and $\Omega_b$
be an $\Omega$ as in the Ultrametric Stable Manifold Theorem,
applied with $a$ and $b$, respectively.
By \cite[Theorem~6.2\,(f)]{EXP} and the proof of
Theorem~1.3 in \cite{EXP},
we may assume that $\Omega_a=\kappa^{-1}(\Gamma_a)$
and $\Omega_b=\kappa^{-1}(\Gamma_b)$,
where
%
\begin{eqnarray}
\Gamma_a & =& \{ z\in B_r^E(0)\colon \mbox{($\forall n\in \N_0$) $g^n(z)$
is defined and $\|g^n(z)\|\leq a^nr$}\}\; \mbox{and}\notag\\
\Gamma_b & =& \{ z\in B_t^E(0)\colon \mbox{($\forall n\in \N_0$) $g^n(z)$
is defined and $\|g^n(z)\|\leq b^nt$}\}\label{graess}
\end{eqnarray}
for certain $r,t>0$.
Moreover, by \cite[Theorem~6.2\,(e)]{EXP},
we may assume that $r=t$, after replacing both $r$ and $t$
by $\min\{r,t\}$.
Then $\Gamma_a\sub \Gamma_b$
by~(\ref{graess}),
and hence $\Gamma_a=\Gamma_b$
(since both sets are graphs of functions on the
same domain, by the cited theorem). 
Thus $\Omega_a=\Omega_b$,
entailing that $W_a^\obs(f,p)=W_b^\obs(f,p)$
as a set and also as an immersed submanifold
of~$M$ (cf.\ proof of \cite[Theorem~1.3]{EXP}).

(c) By~(\ref{soahyp}), we have
$E_{a,\obs}=\bigoplus_{\rho<a}E_\rho=\{0\}$,
whence $\Omega=\kappa^{-1}(\Gamma)=\{p\}$
in \cite[Theorem~1.3]{EXP} and its proof.
Thus $W_a^\obs(f,p)=\bigcup_{n\in \N_0}f^{-n}(\Omega)=\{p\}$.
\end{proof}
\section{Results for automorphisms of Lie groups}\label{seclie}
%
%
Throughout this section,
$G$ is an analytic Lie group
modelled on an ultrametric Banach space
over a complete ultrametric field $(\K,|.|)$,
and $\beta\colon G\to G$ an analytic automorphism.
Then the neutral element $e\in G$ is a fixed point of~$\beta$,
and hence our general theory applies.
We now compile some
additional conclusions which are specific to
automorphisms. Like results of the previous
sections, these are needed for the farther-reaching
Lie-theoretic applications described in the introduction.\\[2.5mm]
We begin with a corollary to Proposition~\ref{unicon}.
An automorphism $\beta\colon G\to G$
is called \emph{contractive}
if $\lim_{n\to\infty}\beta^n(x)=e$ for each $x\in G$.
\begin{cor}
If $G$ is finite-dimensional
and $\beta\colon G\to G$ a contractive\linebreak
automorphism,
then every eigenvalue $\lambda$ of
$ L(\beta) \tensor_\K  \id_{\wb{\K}}$ in an algebraic\linebreak
closure~$\wb{\K}$
has absolute value $|\lambda|<1$.
\end{cor}
\begin{proof}
$G$ is complete by \cite[Proposition~2.1\,(a)]{FOR},
and metrizable. Since every identity neighbourhood $P$ in $G$
contains an open subgroup $U$ of $G$ (see, e.g.,
\cite[Proposition~2.1\,(a)]{FOR}),
Lemma~1\,(a) in \cite{Sie}
provides a $\beta$-invariant open subgroup $Q:=U_{(0)}\sub U\sub P$
of $G$. Hence~$e$ is a uniformly contractive fixed
point of~$\beta$, and thus ``(a)$\impl$(c)'' in Proposition~\ref{unicon}
applies.
\end{proof}
%
%
%
\begin{prop}\label{proprev}
If $a\in \;]0,1]$
and $L(\beta)$ is $a$-hyperbolic,
the following holds:
\begin{itemize}
\item[\rm(a)]
The $a$-stable manifold $W_a^\obs(\beta,e)$ is an immersed
Lie subgroup of~$G$.
\item[\rm(b)]
If, moreover, $L(G)$ admits a centre subspace
with respect to $L(\beta)$
and $L(G)_{1,\obu}=\{0\}$,
then $W_a^\obs(\beta,e)$ is a
Lie subgroup of~$G$.
\end{itemize}
\end{prop}
\begin{proof}
(a) The proof of \cite[Proposition~4.6]{MaZ}
applies without changes.\footnote{In $\diamondsuit$, read ``$\,\leq a^n\,$''
as ``$\,< a^n r$.''}

(b) is a special case of Proposition~\ref{Bangd}.
\end{proof}
If $G$ is finite-dimensional,
then the extra hypotheses
in Proposition~\ref{proprev}\,(b)
mean that $R(L(\beta))\sub \;]0,1]$
(see Theorem~A and (\ref{sodecac})).\\[2.5mm]
In the following situation,
hyperbolicity
is not needed to make $W^\obs$
a manifold.
%
%
%
\begin{prop}\label{notmostgen}
If $\beta\colon G\to G$ is an automorphism
and $L(G)$ admits a centre subspace with respect
to $L(\beta)\colon L(G)\to L(G)$,
then the following holds:
\begin{itemize}
\item[\rm(a)]
There exist a local stable manifold $V_1$
and a centre manifold $V_0$ around~$e$ with respect to~$\beta$,
and a local stable manifold $V_{-1}$ around $e$ with respect
to $\beta^{-1}$,
such that $V_1V_0V_{-1}$ is open in $G$ and the product map
\begin{equation}\label{thepromp}
\pi \colon V_1\times V_0\times V_{-1} \to V_1V_0V_{-1}\,\quad (x,y,z)\mto xyz
\end{equation}
is an analytic diffeomorphism.
\item[\rm(b)]
There is a unique
immersed submanifold structure on
$W^\obs(\beta,e)$
such that
conditions {\rm (a)--(c)}
of the Ultrametric Stable Manifold Theorem $($from the Introduction$)$
are satisfied.
This immersed submanifold structure
makes $W^\obs(\beta,e)$
an immersed Lie subgroup
of~$G$,
and also the final assertion of the cited theorem
holds.
Moreover,
$W^\obs(\beta,e)=W_a^\obs(\beta,e)$
for some $a\in \;]0,1[$ such that
$L(\beta)$ is $a$-hyperbolic.
\end{itemize}
\end{prop}
\begin{proof}
(a) Set $E:=L(G)$ and let $E=E_1\oplus E_0\oplus E_{-1}$
be the decomposition into a
stable subspace $E_1$, 
centre subspace $E_0$ and unstable subspace $E_{-1}$
with respect to $L(\beta)$,
and $\|.\|$ be an ultrametric norm as
in Definition \ref{defacentre}.
There is $a\in \;]0,1[$ such that
$\|L(\beta)|_{E_1}\|<a$
and $\frac{1}{\|L(\beta)^{-1}|_{E_{-1}}\|}>\frac{1}{a}$.
Then $L(\beta)$ is $a$-hyperbolic
with $a$-stable subspace $E_1$
and $a$-unstable subspace
$E_0\oplus E_{-1}$
(and the norm $\|.\|$ as before).
Also $L(\beta)^{-1}$ is $a$-hyperbolic,
with $a$-stable subspace $E_{-1}$
and $a$-unstable subspace
$E_0\oplus E_1$
(and the norm $\|.\|$ as before).
We let
$V_1$ be a local $a$-stable
manifold around~$e$ with respect to~$\beta$
and $V_{-1}$ be a local $a$-stable
manifold around~$e$ with respect to~$\beta^{-1}$
(see \cite[Theorem~6.6\,(a)]{EXP});
by \cite[Theorem 6.6\,(c)]{EXP}, we may assume that
$V_1\sub W_a^\obs(\beta,e)$.
Also, we let $V_0$ be a centre manifold
around $p$ with respect to~$\beta$ (see \cite[Theorem 1.10\,(a)]{EXP}).
Then $T_e(V_1)=E_1$, $T_e(V_0)=E_0$ and $T_e(V_{-1})=E_{-1}$,
whence
\[
L(G)\, =\, T_e(V_1)\oplus T_e(V_0) \oplus T_e(V_{-1})\, .
\]
Thus, after shrinking $V_1$, $V_0$ and $V_{-1}$
(which is possible by \cite[Theorems 6.6\,(c)
and 1.10\,(c)]{EXP}),
we may assume that $P:=V_1V_0V_{-1}$ is open in $G$
and the product map (\ref{thepromp})
is an analytic diffeomorphism
(by the Inverse Function Theorem~\cite{Bo1}).\vspace{1mm}

(b) Shrinking $V_1$, $V_0$ and $V_{-1}$ further
if necessary,
we may assume that there are $r>0$
and charts $\kappa_j\colon V_j\to B^{E_j}_r(0)$
with $\kappa_j(e)=0$ and $d\kappa_j=\id$
for $j\in \{-1,0,1\}$.
There is $s\in \;]0,r]$ such that $\beta(\kappa^{-1}_j(B^{E_j}_s(0)))\sub
V_j$ for all $j\in \{-1,0,1\}$.
Let $g_j:=\kappa_j\circ \beta \circ \kappa_j^{-1}|_{B^{E_j}_s(0)}$.
Shrinking~$s$, we achieve that
\begin{eqnarray}
\|g_0(x)\|\, &=& \,\;\|x\|\hspace*{1.08mm} \quad
\; \mbox{for each $x\in B^{E_0}_s(0)$,}\label{useelswh2}\\
\|g_1(x)\|\,  & < & \; a \|x\| \quad
\, \hspace*{.1mm}\mbox{for each $x\in B^{E_1}_s(0)$, and } \label{useelswh3}\\
\|g_{-1}(x)\| & > & a^{-1}  \|x\| \;\; \mbox{for each $x\in B^{E_{-1}}_s(0)$} \label{useelswh4}
\end{eqnarray}
(using (\ref{flower})).
Then
\[
\kappa:=(\kappa_1\times \kappa_0\times \kappa_{-1})\circ \pi^{-1}\colon
P\to B^E_r(0)
\]
is a chart of~$G$
around~$e$. We set
$g:=g_1\times g_0 \times g_{-1} \colon B^E_s(0)\to
B^E_r(0)$
(where $B^E_s(0)=B^{E_1}_s(0)\times B^{E_0}_s(0) \times B^{E_{-1}}_s(0)$).
Abbreviate $Q:=\kappa^{-1}(B^E_s(0))$.
Then
%
\begin{equation}\label{locconj}
\beta|_Q=\kappa^{-1}\circ g \circ\kappa|_Q\,.
\end{equation}
If $z\in W^\obs (\beta,e)$, there is $n_0\in \N_0$
such that $\beta^n(z)\in Q$ for all $n\geq n_0$,
and
\begin{equation}\label{wllwspr}
\|\kappa(\beta^n(z))\|\to 0\quad\mbox{as $\,n\to\infty$.}
\end{equation}
After replacing $z$ with $\beta^{n_0}(z)$,
we may assume that $n_0=0$.
Now $x=(x_1, x_0, x_{-1}):=\kappa(z)$
is an element of $B_s^E(0)$
such that $g^n(x)=\kappa(\beta^n(z))\in B^E_s(0)$
for all $n\in \N_0$
(cf.\ (\ref{locconj})).
Also
%
\begin{equation}\label{zwstp}
\lim_{n\to\infty}\|g^n(x)\|=0\,,
\end{equation}
by (\ref{wllwspr}).
Since
$\|g^n(x)\|=\max\{\|g_1^n(x_1)\|, \| g_0^n(x_0)\|,\|g_{-1}^n(x_{-1})\|\}$
for all $n\in \N_0$,
using (\ref{useelswh2}) and (\ref{useelswh4})
we obtain a contradiction to~(\ref{zwstp})
unless $x_0=0$ and $x_{-1}=0$.
Thus $x=x_1\in E_1$
and thus $z=\kappa_1^{-1}(x_1)\in V_1\sub W_a^\obs(\beta,e)$,\linebreak
entailing that $W^\obs(\beta,e)\sub W_a^\obs(\beta,e)$.
The converse inclusion being trivial,
we deduce that
$W^\obs(\beta,e)
=W_a^\obs(\beta,e)$.
We give $W^\obs(\beta,e)$
the manifold structure of
$W_a^\obs(\beta,e)$.
It then is tangent to $E_{a,\obs}=E_1$ at~$e$.
Hence $W^\obs(\beta,e)$
satisfies conditions (a)--(c)
of the Ultrametric Stable Manifold Theorem
and also the final assertion of the theorem.
To obtain the uniqueness of the immersed submanifold
structure subject to these conditions,
note that for any such structure on $W^\obs$,
each neighbourhood of $e$ in $W^\obs$ contains
an open $\beta$-invariant neighbourhood of~$e$
(as this only requires (\ref{domi})
and \ref{remstrict}).
Now one shows as in the proof of \cite[Theorem~6.6\,(b)]{EXP}
that the germ of the latter
coincides with the germ we\linebreak
already have,
and this entails as in the proof of the uniqueness
part of\linebreak
\cite[Theorem~1.3]{EXP}
that the new manifold structure on $W^\obs$
coincides with the one we already had
(further explanations are omitted,
because the assertion is not central).
All other assertions follow from
Proposition~\ref{proprev}.
\end{proof}
{\bf Proof of Theorem D.}
We now prove Theorem D.
The proof will provide
additional information:
\emph{$W^\obs(\beta,e)=W_a^\obs(\beta,e)$
for each $a\in \;]0,1[$ such that
$[a,1[\,\cap R(L(\beta))=\emptyset$
and $\;]1,\frac{1}{a}]\cap R(L(\beta))=\emptyset$}.\\[2.5mm]
If we choose $\|.\|$ as a norm adapted to $L(\beta)$ (as in
Definition \ref{defnadpt})
in the proof of Proposition \ref{notmostgen},
then $E_1$, $E_0$ and $E_{-1}$
are the direct sum of all $L(G)_\rho$
with $\rho\in R(L(\beta))$,
such that $\rho\in \;]0,1[$ (resp., $\rho=1$,
resp., $\rho\in \;]1,\infty[$), by (\ref{sodecac}).
If $a$ is as described at the beginning of the proof,
then $\|L(\beta)\|<a$ and $\|L(\beta)^{-1}\|<a$
(as is clear from (b) and (c) in Definition \ref{defnadpt}).
Therefore the proof of Proposition \ref{notmostgen}
applies with this choice of $a$.\,\Punkt
{\footnotesize
Helge Gl\"{o}ckner, Universit\"at Paderborn,
Institut f\"ur Mathematik,
Warburger Str.\ 100,\\
33098 Paderborn, Germany. E-Mail: glockner\at{}math.upb.de}
\end{document}